\documentclass[10pt]{amsart}

\usepackage{enumerate}
\usepackage{amsmath,amssymb}

\newtheorem{theorem}{Theorem}[section]
\newtheorem*{thm}{Theorem A}
\newtheorem*{thmb}{Theorem B}
\newtheorem*{q}{Question}
\newtheorem{corollary}[theorem]{Corollary}
\newtheorem{lemma}[theorem]{Lemma}

\newtheorem{proposition}[theorem]{Proposition}

\newtheorem{conjecture}[theorem]{Conjecture}

\theoremstyle{definition}
\newtheorem{definition}[theorem]{Definition}
\newtheorem{remark}[theorem]{Remark}

\def\C{{\mathbb{C}}}
\def\M{{\mathbb{M}}}
\def\N{{\mathbb{N}}}
\def\R{{\mathbb{R}}}

\def\cC{{\mathcal{C}}}
\def\P{{\mathcal{P}}}

\def\Z{{\mathcal{Z}}}

\makeatletter
  
  \@addtoreset{equation}{section}
\makeatother

\usepackage{comment}

\newcommand{\vertiii}[1]{{\left\vert\kern-0.25ex\left\vert\kern-0.25ex\left\vert #1 
    \right\vert\kern-0.25ex\right\vert\kern-0.25ex\right\vert}}

\begin{document}

\title[Lattice isomorphisms]{Lattice isomorphisms between projection lattices of von Neumann algebras}

\author[M. Mori]{Michiya Mori}

\address{Graduate School of Mathematical Sciences, The University of Tokyo, 3-8-1 Komaba, Tokyo, 153-8914, Japan.}
\email{mmori@ms.u-tokyo.ac.jp}



\date{}

\begin{abstract}
Generalizing von Neumann's result on type II$_1$ von Neumann algebras, we characterize lattice isomorphisms between projection lattices of arbitrary von Neumann algebras by means of ring isomorphisms between the algebras of locally measurable operators. 
Moreover, we give a complete description of ring isomorphisms of locally measurable operator algebras when the von Neumann algebras are without type II direct summands. 
\end{abstract}

\maketitle
\thispagestyle{empty}

\section{Introduction}
Since the very first work \cite{MN} by Murray and von Neumann more than 80 years ago, the geometry of projections has played the central role in understanding the structure of von Neumann algebras (rings of operators). 
For a von Neumann algebra $M$, let $\P(M)$ denote the projection lattice of $M$, that is, $\P(M):=\{p\in M\mid p=p^*=p^2\}$. 
In this paper, we would like to consider the following question: 
What is the general form of lattice isomorphisms between projection lattices of von Neumann algebras? 

There are several important results related to this question. 
Let us first think about finite dimensional factors. 
The case $M=N=\M_n(\C)$ for $n=1, 2$ is not interesting at all. 
Indeed, if $n=1$, then $\P(\M_n(\C))$ is $\{0, 1\}$, and a lattice automorphism of it is the identity mapping. 
If $n=2$, then a bijection $\Phi$ on $\P(\M_n(\C))$ is a lattice automorphism if and only if $\Phi(0)=0$ and $\Phi(1)=1$. 
If $M=N= \M_n(\C)$ for $3\leq n<\infty$, then the fundamental theorem of projective geometry gives an answer to our question.
Recall that a function $f\colon X\to Y$ between complex vector spaces is said to be \emph{semilinear} if it is additive and there exists a ring homomorphism $\sigma\colon \C\to \C$ satisfying $f(cx) = \sigma(c)f(x)$ for all $c\in \C$ and $x\in X$. 
\begin{theorem}[Fundamental theorem of projective geometry]
Let $3\leq n<\infty$. 
Suppose that $\Phi\colon \P(\M_n(\C))\to \P(\M_n(\C))$ is a lattice isomorphism. 
Then there exists a semilinear bijection $f\colon \C^n\to \C^n$ such that $\Phi(p_{\xi}) = p_{f(\xi)}$ for every $\xi\in \C^n$, where $p_{\xi}$ denotes the projection from $\C^n$ onto $\C \xi$ for a vector $\xi\in \C^n$.  
\end{theorem}

In the case of type I$_{\infty}$ factors, we can make use of a result below by Fillmore and Longstaff in 1984. 
Recall that a projection $p\in \P(B(H))$ can be identified with its range $pH$, which is a closed subspace of $H$. 
\begin{theorem}[{\cite[Theorem 1]{FL}}]
Let $X$ and $Y$ be infinite dimensional complex normed spaces. 
Let $\cC(X)$ (resp.\ $\cC(Y)$) denote the lattice of all closed subspaces of $X$ (resp.\ $Y$), ordered by inclusion. 
Suppose that $\Phi\colon \cC(X)\to \cC(Y)$ is a lattice isomorphism.
Then there exists a bicontinuous linear or conjugate-linear bijection $f\colon X\to Y$ such that $\Phi(C)=f(C)$ for any $C\in \cC(X)$. 
\end{theorem}

See also the classical result \cite[Theorem 1]{KM}, in which Kakutani and Mackey studied orthocomplementation on the lattice $\P(B(H))$. 

For type I factors, we may observe a correspondence between lattices and rings. 
Let $H$ be a Hilbert space with $\dim H\geq 3$. 
For any lattice automorphism $\Phi\colon \P(B(H))\to \P(B(H))$,
take a mapping $f\colon H\to H$ as above. 
It is a semilinear bijection if $\dim H<\infty$; a linear or conjugate-linear bounded bijection if $\dim H=\infty$. 
Hence we may construct a ring automorphism $\Psi\colon B(H)\to B(H)$ such that $\Phi(l(x))= l(\Psi(x))$ for every $x\in B(H)$ (namely, $\Psi(x) := f\circ x\circ f^{-1}$), where $l(x)$ denotes the left support projection of $x$.  
It is easy to see that the converse also holds. 
That is, any ring automorphism $\Psi\colon B(H)\to B(H)$ determines a lattice automorphism $\Phi$ of $\P(B(H))$ such that $\Phi(l(x))= l(\Psi(x))$ for every $x\in B(H)$. 

We next consider finite von Neumann algebras. 
In 1930's, motivated by the geometry of projection lattices of type II$_1$ factors, 
von Neumann produced the beautiful theory on the correspondence between complemented modular lattices and regular rings. 
One of his achievements \cite[Part II, Theorem 4.2]{N}, applied to the case of arbitrary type II$_1$ von Neumann algebras, reads as follows. 
\begin{theorem}[von Neumann]\label{vN}
Let $M$ and $N$ be von Neumann algebras of type II$_1$. 
Suppose that $\Phi\colon \P(M)\to \P(N)$ is a lattice isomorphism. 
Then there exists a unique ring isomorphism $\Psi\colon S(M)\to S(N)$ between the algebras of measurable operators such that $\Phi(l(x))=l(\Psi(x))$ for any $x\in S(M)$.
\end{theorem}
See the next section for the definition of undefined terms and see also Section \ref{question} for further details about von Neumann's theory.   

In the general setting of von Neumann algebras, with an additional assumption, Dye obtained the following result in 1955. 
\begin{theorem}[{\cite[Corollary of Theorem 1]{Dy}}, see also {\cite[Theorem 1]{Fe}}]\label{dye}
Let $M$ and $N$ be von Neumann algebras without type I$_2$ direct summands. 
Suppose that $\Phi\colon \P(M)\to \P(N)$ is a lattice isomorphism with 
\[
pq=0 \iff \Phi(p)\Phi(q)=0
\]
for any $p, q\in \P(M)$. 
Then there exists a real $^*$-isomorphism $\Psi\colon M\to N$ that extends $\Phi$. 
\end{theorem}

Each of the above results implies that lattice isomorphisms between projection lattices are closely related to ring isomorphisms. 
See also McAsey's survey \cite{Mc} which discusses projection lattice isomorphisms in various settings.   
It is natural to imagine that we can give a similar result for arbitrary lattice isomorphisms in the general setting of von Neumann algebras. 
The main theorem of this paper realizes it. 
\begin{thm}
Let $M$ and $N$ be two von Neumann algebras.  
Suppose that $M$ does not admit type I$_1$ nor I$_2$ direct summands, and that $\Phi\colon \P(M)\to \P(N)$ is a lattice isomorphism. 
Then there exists a unique ring isomorphism $\Psi\colon LS(M)\to LS(N)$ such that $\Phi(l(x)) = l(\Psi(x))$ for all $x\in LS(M)$. 
\end{thm}
Here, $LS(M)$ and $LS(N)$ mean the algebras of locally measurable operators of $M$ and $N$, respectively (see Section \ref{LSdef}). 
We remark that the converse of Theorem A can be verified without difficulty. 
Namely, any ring isomorphism $\Psi\colon LS(M)\to LS(N)$ determines a unique lattice isomorphism $\Phi\colon \P(M)\to \P(N)$ such that $\Phi(l(x)) = l(\Psi(x))$ for all $x\in LS(M)$ (Proposition \ref{converse}). 
Therefore, Theorem A naturally gives rise to the following 
\begin{q}
Let $M$, $N$ be von Neumann algebras. 
What is the general form of ring isomorphisms from $LS(M)$ onto $LS(N)$?
\end{q}
We may answer this question for type I von Neumann algebras using ring isomorphisms of their centers (Proposition \ref{I}). 
Moreover, we obtain  
\begin{thmb}
Let $M, N$ be von Neumann algebras of type I$_{\infty}$ or III. 
If $\Psi\colon LS(M)\to LS(N)$ is a ring isomorphism, then there exist a real $^*$-isomorphism $\psi\colon M\to N$ (which extends to a real $^*$-isomorphism from $LS(M)$ onto $LS(N)$) and an invertible element $y\in LS(N)$ such that  $\Psi(x)=y\psi(x)y^{-1}$, $x\in LS(M)$.
\end{thmb}
We leave the case of type II von Neumann algebras as an open question. 

In Section \ref{preliminaries}, we introduce some tools we use later. 
Section \ref{main} is devoted to the proof of Theorem A. 
The proof is based on the combination of von Neumann's strategy in \cite[Part II, Chapter IV]{N} and a binary relation on the projection lattice which we call LS-orthogonality. 
After that we give a proof of Dye's theorem as an application of Theorem A.
We consider Question in Section \ref{extra}, and prove Theorem B.
This paper ends with comparison of our result with von Neumann's theory and several suggestions of further research directions (Section \ref{question}).

\section{Preliminaries}\label{preliminaries}
Let $M\subset B(H)$ be a von Neumann algebra. 
We use the symbols $\sim$ to mean the Murray--von Neumann equivalence relation on $\P(M)$. 
That is, for $p, q\in \P(M)$, $p\sim q$ means that there exists a partial isometry $v\in M$ such that $p= vv^*$ and $q= v^*v$. 
As usual, for $p, q\in \P(M)$, $p\perp q$ means that $p$ and $q$ are orthogonal. That is, $pq=qp=0$, or equivalently, $pH\perp qH$ in the Hilbert space $H$. 
We use the symbol $p^{\perp}:= 1-p$ for $p\in \P(M)$. 
The symbol $\Z(M)=\{x\in M\mid xy=yx\text{ for all }y\in M\}$ means the center of $M$.

For $n\in \N=\{1, 2, \ldots\}$, we say that $M$ has \emph{order $n$} if there exists a collection $p_1, \ldots, p_n$ of mutually orthogonal projections in $M$ such that $p_1\sim p_2\sim\cdots\sim p_n$ and $\sum_{k=1}^n p_k= 1$. 
It is well known that every von Neumann algebra without finite type I direct summands has order $n$ for any $n\in \N$ \cite[Lemma 6.5.6]{KR}. 
In particular, such an algebra has order $3$. 
It follows that every von Neumann algebra $M$ without type I$_1$ and I$_2$ direct summands can be decomposed into the ($\ell^{\infty}$-)direct sum of von Neumann algebras $M_n$, $3\leq n<\infty$, such that $M_n$ has order $n$ for every $n$. 
If $M$ has order $n\in \N$, then $M$ can be identified with the algebra $\M_n(\hat{M})$ of $n\times n$ matrices with entries in some von Neumann algebra $\hat{M}$.

\subsection{Various isomorphisms of von Neumann algebras}
For $^*$-algebras $A$ and $B$, a (not necessarily linear) bijection $\psi\colon A\to B$ is called
\begin{itemize}
\item a \emph{semigroup isomorphism} if it is multiplicative,
\item a \emph{ring isomorphism} if it is additive and multiplicative,
\item a \emph{real algebra isomorphism} if it is a real-linear ring isomorphism, 
\item an \emph{algebra isomorphism} if it is a complex-linear ring isomorphism, 
\item a \emph{real $^*$-isomorphism} if it is a real algebra isomorphism and satisfies $\psi(x^*)=\psi(x)^*$ for any $x\in A$, 
\item a \emph{$^*$-isomorphism} if it is a complex-linear real $^*$-isomorphism, and
\item a \emph{conjugate-linear $^*$-isomorphism} if it is a conjugate-linear real $^*$-isomorphism. 
\end{itemize}

\begin{lemma}\label{isoms}
Let $M$ and $N$ be von Neumann algebras. 
Suppose that $\psi\colon M\to N$ is a bijection. 
\begin{enumerate}
\item If $M$ is without type I$_1$ direct summands and $\psi$ is a semigroup isomorphism, then $\psi$ is a ring isomorphism. 
\item If $M$ does not admit a finite dimensional ideal and $\psi$ is a ring isomorphism, then $\psi$ is a real algebra isomorphism. 
\item If $\psi$ is a real algebra isomorphism, then there exist a real $^*$-isomorphism $\psi_0\colon M\to N$ and an invertible element $y\in N$ such that $\psi(x) =y\psi_0(x)y^{-1}$ for any $x\in M$. 
\item If $\psi$ is a real $^*$-isomorphism, then there exist central projections $p\in M$ $q\in N$, a $^*$-isomorphism $\psi_1\colon Mp\to Nq$, and a conjugate-linear $^*$-isomorphism $\psi_2\colon Mp^{\perp}\to Nq^{\perp}$ such that $\psi(x) = \psi_1(xp) + \psi_2(xp^{\perp})$ for any $x\in M$. 
\end{enumerate}
\end{lemma}
\begin{proof}
Each item is easily obtained by known results.

(1) We may take a projection $p\in \P(M)$ such that both of the central supports of $p$ and $1-p$ are equal to $1$. 
It is easy to see that the following hold:
(a) If $x\in M$ satisfies $xM=\{0\}$, then $x=0$; 
(b) If $x\in M$ satisfies $pMx=\{0\}$, then $x=0$;
(c) If $x\in M$ satisfies $pxpMp^{\perp}=\{0\}$, then $pxp=0$. 
Hence we may apply Martindale's theorem \cite[Theorem]{Ma} to obtain the desired conclusion.

The item (2) is a consequence of Kaplansky's result \cite[Theorem]{Ka}.

We prove (3) and (4) at the same time. 
Let $\psi\colon M\to N$ be a real algebra isomorphism.
We know that $\psi(i)^2 =\psi(i^2)=\psi(-1)= -1$ and that $\psi(i)$ is central in $N$. 
It follows that $\psi(i) = qi-q^{\perp}i$ for some central projection $q$ of $N$. 
Put $p:= \psi^{-1}(q)$, which is a central projection of $M$. 
If $\psi$ is a real $^*$-isomorphism, then $\psi$ restricted to $Mp$ is a $^*$-isomorphism from $Mp$ onto $Nq$, and $\psi$ restricted to $Mp^{\perp}$ is a conjugate-linear $^*$-isomorphism from $Mp^{\perp}$ onto $Nq^{\perp}$, hence the proof of (4) is complete. 
If $\psi$ is merely a real algebra isomorphism, then $\psi$ restricted to $Mp$ is an algebra isomorphism from $Mp$ onto $Nq$, and $\psi$ restricted to $Mp^{\perp}$ determines an algebra isomorphism from $Mp^{\perp}$ onto $\overline{Nq^{\perp}}$, where $\overline{Nq^{\perp}}$ means the complex conjugation of the von Neumann algebra $Nq^{\perp}$.  
See e.g.\ \cite[Section 2.3]{P} for the definition of complex conjugation of von Neumann algebras.
Lastly, we may use the result on the general form of algebra isomorphisms between von Neumann algebras (\cite[Theorem I]{O}, see also \cite{Ga} and \cite[Section 4.1]{Sa}) to obtain the desired conclusion. 
\end{proof}

\subsection{The algebra of locally measurable operators}\label{LSdef}
Let $M\subset B(H)$ be a von Neumann algebra. 
In this paper, the algebra $LS(M)$ of locally measurable operators with respect to $M$, which we briefly describe below, plays a crucial role.

A densely defined closed operator $x$ on $H$ is said to be \emph{affiliated with $M$} (and we write $x\eta M$) if 
$yx\subset xy$ for any $y\in M'$, where $M':=\{y\in B(H)\mid ay=ya\text{ for any }a\in M\}$ denotes the commutant of $M$. 
An operator $x\eta M$ is said to be \emph{measurable} with respect to $M$ if the spectral projection $\chi_{(c, \infty)}(\lvert x\rvert)\in \P(M)$ is a finite projection in $M$ for some real number $c>0$. 
An operator $x\eta M$ is said to be \emph{locally measurable} with respect to $M$ if there exists an increasing sequence $\{p_n\}_{n\geq 1}$ of central projections in $M$ such that $p_n\nearrow 1$ and $xp_n$ is measurable with respect to $M$ for any $n$.
We write $S(M)$ (resp.\ $LS(M)$) to mean the collection of all measurable (resp.\ locally measurable) operators with respect to $M$.
If $x, y\in S(M)$ (resp. $LS(M)$), then $x^*$ and the closures of $xy$, $x+y$ are in $S(M)$ (resp. $LS(M)$). 
Using this fact, we can consider $S(M)$ and $LS(M)$ as $^*$-algebras that contain $M$. 
In what follows, we abbreviate the symbol of the closure of an unbounded operator unless it is confusing.
We remark that $LS(M)=M$ holds if and only if $M$ is the direct sum of finite number of type I and III factors. 
We also remark that if $M$ is finite then $LS(M)=S(M)$ is the collection of all affiliated operators.
See \cite{Y} and \cite{Se} for more details of (locally) measurable operators. 

In \cite[Lemma 2.2]{Mor2}, the author obtained the following result. 

\begin{lemma}\label{inverse}
Let $M$ be a von Neumann algebra and $a\in M_+$. 
Then the following two conditions are equivalent: 
\begin{enumerate}
\item The element $a$ is invertible in the algebra $LS(M)$. 
\item For any $b\in M_+\setminus \{0\}$, there exists an $x\in M_+\setminus \{0\}$ such that $x\leq a$ and $x\leq b$. 
\end{enumerate}
\end{lemma}

For $x\in LS(M)$, let $l(x)\in \P(M)$ denote the left support of $x$. 
That is, $l(x) :=\bigwedge \{p\in \P(M)\mid px = x\}$. 
Similarly, we write $r(x) :=\bigwedge \{p\in \P(M)\mid x = xp\}$.
Then $l(x)= \chi_{(0, \infty)}(\lvert x^*\rvert)$ and $r(x)= \chi_{(0, \infty)}(\lvert x\rvert)$ hold.
We remark that, for $x, y\in LS(M)$, we have $xy=0$ if and only if $r(x)l(y)=0$. 
Indeed, if $r(x)l(y)=0$, then $xy=xr(x)l(y)y=0$.
If $xy=0$, then we have $\lvert x\rvert \lvert y^*\rvert = 0$, 
which implies $\chi_{(\varepsilon, \infty)}(\lvert x\rvert)\chi_{(\varepsilon, \infty)}(\lvert y^*\rvert)=0$ for every $\varepsilon>0$. 
Take the limit $\varepsilon\to 0$ in the strong operator topology to obtain   $r(x)l(y)=0$.

\subsection{Halmos's two projection theorem}\label{halmos}
In order to play with projection lattices, it is useful to look at the relative position of a pair of projections. 
For it, we make use of Halmos's two projection theorem \cite{H} from the viewpoint of von Neumann algebra theory. 
Here we recapitulate the argument in \cite[Lemma 2.2]{Mor1}.

Let $M\subset B(H)$ be a von Neumann algebra and $p, q\in \P(M)$. 
Put 
\[
e_1= p- p\wedge q - p\wedge q^{\perp},\quad e_2=p^{\perp} - p^{\perp}\wedge q - p^{\perp}\wedge q^{\perp}, 
\]
and $x:= e_1(q-p\wedge q-p^{\perp}\wedge q)e_2$.
By an elementary calculation, we see that $l(x)=e_1$ and $r(x)=e_2$. 
By polar decomposition, we may take a partial isometry $v=v_{p, q}\in M$ such that $x= v\lvert x \rvert = \lvert x^* \rvert v$, $vv^*= e_1$ and $v^*v= e_2$. 

We can identify each $y\in (e_1+e_2)M(e_1+e_2)$ with 
$\begin{pmatrix}
e_1ye_1 & e_1yv^* \\
vye_1 & vyv^*
\end{pmatrix} \in \M_2(e_1Me_1)$. 
Then $q-p\wedge q-p^{\perp}\wedge q\, (\leq e_1+e_2)$ is identified with 
\[
\begin{split}
&\begin{pmatrix}
e_1(q-p\wedge q-p^{\perp}\wedge q)e_1 & e_1(q-p\wedge q-p^{\perp}\wedge q)v^* \\
v(q-p\wedge q-p^{\perp}\wedge q)e_1 & v(q-p\wedge q-p^{\perp}\wedge q)v^*
\end{pmatrix} \\
=& \begin{pmatrix}
e_1(q-p\wedge q-p^{\perp}\wedge q)e_1 & \lvert x^* \rvert \\
\lvert x^* \rvert & v(q-p\wedge q-p^{\perp}\wedge q)v^*
\end{pmatrix}
\in \M_2(e_1Me_1).
\end{split}
\] 
Put $a:= (e_1(q-p\wedge q-p^{\perp}\wedge q)e_1)^{1/2}$ and $b:= (v(q-p\wedge q-p^{\perp}\wedge q)v^*)^{1/2}$, which are positive injective operators in $M_{p, q}:= e_1Me_1$. 
Since $\begin{pmatrix}
a^2 & \lvert x^* \rvert \\
\lvert x^* \rvert & b^2
\end{pmatrix}$ is a projection, some calculations show that $a, b$ and $\lvert x^* \rvert$ commute, $a^2+b^2= e_1$ and $\lvert x^* \rvert = ab$. 
Thus $q-p\wedge q-p^{\perp}\wedge q$ corresponds to $\begin{pmatrix}
a^2 & ab \\
ba & b^2
\end{pmatrix}$.

Therefore, we may decompose $p$ and $q$ in the following manner:
\[
p= 1\oplus 0\oplus 1\oplus 0\oplus 
\begin{pmatrix}
1&0\\
0&0
\end{pmatrix},\quad 
q= 0\oplus 1\oplus 1\oplus 0\oplus 
\begin{pmatrix}
a^2&ab \\
ab&b^2
\end{pmatrix}, 
\]
where $H$ is decomposed as $H = (p\wedge q^{\perp}) H \oplus (p^{\perp} \wedge q) H \oplus (p\wedge q) H \oplus (p^{\perp} \wedge q^{\perp}) H \oplus (e_1+e_2)H$, 
and $a$ and $b$ are positive injective operators in $M_{p, q}\,\, (= e_1Me_1)$ such that $a^2+b^2=1_{M_{p, q}}$.

\subsection{Center-valued norm}\label{cvn}
Let $M$ be a von Neumann algebra of type I or III and $x\in LS(M)$. 
Then there exists a unique minimal element $\vertiii{x} \in LS(\Z(M))_+\,\,(\subset LS(M))$ with $\lvert x\rvert\leq \vertiii{x}$. 
The mapping $\vertiii{\cdot}\colon LS(M)\to LS(\Z(M))_+$ is called the \emph{center-valued norm}. 
Remark that if $M$ is a factor, then $\Z(M)$ can be identified with $\C$ and we have $\vertiii{x} = \lVert x \rVert\in \R$ for every $x\in M$. 
Be cautious of the fact that we cannot take such a mapping for a type II von Neumann algebra. 
That's why we will need to exclude type II cases in the proof of Theorem B.

As is expected, the center-valued norm possesses e.g.\ the following properties: 
For any $x, y\in LS(M)$ and $a\in LS(\Z(M))$, we have (i) $\vertiii{x}=0\Longrightarrow x=0$. (ii) $\vertiii{x+y} \leq \vertiii{x}+\vertiii{y}$. (iii) $\vertiii{a} = \lvert a\rvert$. (iv) $\vertiii{ax} = \lvert a\rvert\vertiii{x}$. (v) $\vertiii{xy}\leq \vertiii{x}\vertiii{y}$. 
See for example \cite[Section 2]{AAKD} and references therein for further information about the center-valued norm.

\section{Lattice isomorphisms of projection lattices}\label{main}
Part of this section heavily depends on von Neumann's argument in \cite[Part II, Chapter IV]{N}.
The aim of this section is to give a proof of 
\begin{thm}
Let $M$ and $N$ be two von Neumann algebras.  
Suppose that $M$ does not admit type I$_1$ nor I$_2$ direct summands, and that $\Phi\colon \P(M)\to \P(N)$ is a lattice isomorphism. 
Then there exists a unique ring isomorphism $\Psi\colon LS(M)\to LS(N)$ such that $\Phi(l(x)) = l(\Psi(x))$ for all $x\in LS(M)$. 
\end{thm}

Before its proof, we consider the converse of Theorem A.
\begin{proposition}\label{converse}
Let $M$ and $N$ be von Neumann algebras. 
Suppose that $\Psi\colon LS(M)\to LS(N)$ is a ring isomorphism. 
Then there exists a unique lattice isomorphism $\Phi\colon \P(M)\to \P(N)$ such that $\Phi(l(x)) = l(\Psi(x))$ for any $x\in LS(M)$. 
\end{proposition}
\begin{proof}
It is easy to see that $\Psi(0)=0$.
Let $x, y\in LS(M)$ satisfy $l(x)\leq l(y)$. 
Then we have $\{z\in LS(M)\mid zx\neq 0\}\subset \{z\in LS(M)\mid zy \neq 0\}$ and hence $\{z\in LS(N)\mid z\Psi(x)\neq 0\}\subset \{z\in LS(N)\mid z\Psi(y) \neq 0\}$, which in turn leads to $l(\Psi(x))\leq l(\Psi(y))$. 
We obtain $l(x)\leq l(y)\Longleftrightarrow l(\Psi(x))\leq l(\Psi(y))$ for any $x, y\in LS(M)$. 
Therefore, the mapping $\Phi\colon \P(M)\to \P(N)$ defined by $\Phi(p) = l(\Psi(p))$, $p\in \P(M)$, satisfies the desired condition. 
\end{proof}

\begin{remark}
The same proof is valid even if we replace a ring isomorphism with a semigroup isomorphism. 
However, Martindale's result \cite{Ma} implies that a semigroup isomorphism $\Psi\colon LS(M)\to LS(N)$ is automatically a ring isomorphism if $M$ is without type I$_1$ direct summands.
\end{remark}

We begin the proof of Theorem A. 
Let us first check the uniqueness of $\Psi$. 

\begin{lemma}
Let $M$ be a von Neumann algebra without type I$_1$ direct summands. 
For any $x\in M$, there exists a subset $F\subset M$ with $\# F\leq 9$, $\sum_{y\in F} y = x$, and the following property: 
For any $y\in F$, there exists a pair $p, q\in \P(M)$ of mutually orthogonal projections such that $p\sim q$ and either $pyp=y$ or $pyq=y$. 
\end{lemma}
\begin{proof}
It suffices to consider the case where $M$ has fixed order $2\leq n<\infty$. 
Then we may identify $M$ with $\M_n(\hat{M})$ for some von Neumann algebra $\hat{M}$. 
We may write $x\in M$ as $x = (x_{ij})_{1\leq i, j\leq n} \in \M_n(\hat{M})$. 
It is easy to see that we can take integers $n_0:=0\leq n_1\leq n_2\leq n=:n_3$ such that $n_1, n_2-n_1, n_3-n_2\leq n/2$. 
For $1\leq k, l\leq 3$, define $x^{kl} = (x^{kl}_{ij})_{1\leq i, j\leq n}\in \M_n(\hat{M})$ by 
$x^{kl}_{ij} = x_{ij}$ if $n_{k-1}+1\leq i\leq n_k$ and $n_{l-1}+1\leq j\leq n_l$, and $x^{kl}_{ij} = 0$ otherwise. 
(Here, we are decomposing $x$ into $3\times 3$ blocks.)
Then the nine operators $x^{kl}$, $1\leq k, l\leq 3$, (some of which may be $0$) satisfy the desired condition. 
\end{proof}

\begin{lemma}\label{id}
Let $M$ be a von Neumann algebra without type I$_1$ direct summands. 
Suppose that $\Psi\colon LS(M)\to LS(M)$ is a ring isomorphism with $l(\Psi(x))=l(x)$ for all $x\in LS(M)$. 
Then $\Psi$ is the identity mapping on $LS(M)$. 
\end{lemma}
\begin{proof}
Let $p\in \P(M)$. 
We prove $\Psi(p)=p$. 
Since $pp^{\perp} =0$, we have $\Psi(p)\Psi(p^{\perp})=0$, which implies $0=r(\Psi(p))l(\Psi(p^{\perp}))=r(\Psi(p))p^{\perp}$. 
We obtain $r(\Psi(p))\leq p$.
We also have the equation $\Psi(p)^2=\Psi(p^2)=\Psi(p)$. 
Hence we obtain $(p-\Psi(p))\Psi(p)=0$, which implies $0=(p-\Psi(p))l(\Psi(p))=(p-\Psi(p))p$ and $p-\Psi(p)=0$.

In what follows, let $p, q\in \P(M)$ be mutually orthogonal mutually Murray--von Neumann equivalent projections. 
We next prove that $\Psi(x)=x$ if $x\in M\,\,(\subset LS(M))$ satisfies $pxq=x$.
By additivity, we may assume $\lVert x\rVert\leq 1/2$. 
Then there exists a projection $e\in \P(M)$ such that $e\leq p+q$, $peq=x$. 
Indeed, let $x=v\lvert x\rvert=\lvert x^*\rvert v$ be the polar decomposition. 
Take an operator $a\in(pMp)_+$ such that $\lVert a\rVert\leq \pi/4$ and $\lvert x^*\rvert=\sin a\cos a=(\sin 2a)/2$. 
Then 
\[
e:= \cos^2 a + v^*(\sin a\cos a) + (\sin a\cos a)v + v^*(\sin^2 a)v
\]
satisfies this property.
We obtain $\Psi(x)=\Psi(peq)=\Psi(p)\Psi(e)\Psi(q)=peq=x$. 

Suppose that $x\in M$ satisfies $pxp=x$. 
Take a partial isometry $v\in M$ such that $vv^*=p$ and $v^*v=q$. 
Then we have $p(xv)q = xv$ and $qv^*p=v^*$. 
Hence $\Psi(x)=\Psi(xvv^*)=\Psi(xv)\Psi(v^*)= xvv^* = x$. 

By the additivity of $\Psi$ and the preceding lemma, we see that $\Psi$ fixes every element in $M$. 
Let $x\in LS(M)$ and let $x=v\lvert x\rvert$ be its polar decomposition. 
It is clear that $\Psi(1)=1$. 
Since $v, (\lvert x\rvert + 1)^{-1}\in M$, we obtain 
\[
\begin{split}
\Psi(x) = \Psi(v\lvert x\rvert) &= \Psi(v)\Psi(\lvert x\rvert)\\
&=v(\Psi(\lvert x\rvert + 1)-1)= v(\Psi((\lvert x\rvert + 1)^{-1})^{-1} -1)\\
&= v((\lvert x\rvert + 1)-1) = v\lvert x\rvert=x 
\end{split}
\] 
\end{proof}

Hence we obtain the uniqueness of $\Psi$ in Theorem A.
Indeed, if two ring isomorphisms $\Psi, \Psi'\colon LS(M)\to LS(N)$ satisfies $l(\Psi(x))=l(\Psi'(x))$ for all $x\in LS(M)$, then we have $l(\Psi^{-1}\circ\Psi'(x))=l(x)$ for all $x\in LS(M)$, hence the preceding proposition implies $\Psi^{-1}\circ\Psi'(x)=x$ for all $x\in LS(M)$.
\medskip  

We introduce a binary relation on $\P(M)$, which is a key to the proof of Theorem A.
Let $p, q\in \P(M)$ be two projections with $p\wedge q = 0$. 
By Subsection \ref{halmos}, we decompose $p$ and $q$: 
\begin{equation}\label{eq:halmos}
p= 1\oplus 0\oplus 0\oplus 
\begin{pmatrix}
1&0\\
0&0
\end{pmatrix},\quad 
q= 0\oplus 1\oplus 0\oplus 
\begin{pmatrix}
a^2&ab \\
ab&b^2
\end{pmatrix}.    
\end{equation}
We say that $p$ is \emph{LS-orthogonal} to $q$ if the operator $b\in M_{p, q}$ is  invertible in $LS(M_{p, q})$. 

\begin{lemma}\label{orthogonalize}
Let $M$ be a von Neumann algebra and $p, q\in \P(M)$. 
Suppose that $p$ is LS-orthogonal to $q$. 
Then there exists an invertible element $S= S_{p, q}\in LS(M)$ such that $S(p\vee q)^{\perp}=(p\vee q)^{\perp}S= (p\vee q)^{\perp}$, $Sp= p$ and $l(SqS^{-1})=p\vee q - p$. 
\end{lemma}
\begin{proof}
Put $S:= 1\oplus1\oplus1\oplus 
\begin{pmatrix}
1&-ab^{-1} \\
0&b^{-1}
\end{pmatrix}$ with respect to the decomposition as above. 
Then $S$ is an element in $LS(M)$ with inverse $S^{-1}= 1\oplus1\oplus1\oplus 
\begin{pmatrix}
1&a \\
0&b
\end{pmatrix}$. 
It is easy to see that 
\[
S(p\vee q)^{\perp}=(p\vee q)^{\perp}S= (p\vee q)^{\perp} = 0\oplus0\oplus1\oplus 
\begin{pmatrix}
0&0 \\
0&0
\end{pmatrix}.
\] 
We also have 
\[
Sp = 1\oplus 0\oplus 0\oplus \begin{pmatrix}
1&0 \\
0&0
\end{pmatrix} = p
\]
and 
\[
l(SqS^{-1}) = l\left(0\oplus 1\oplus 0\oplus \begin{pmatrix}
0&0 \\
a&1
\end{pmatrix}\right) = 0\oplus 1\oplus 0\oplus \begin{pmatrix}
0&0 \\
0&1
\end{pmatrix} = p\vee q- p. 
\]
\end{proof}

\begin{lemma}\label{LS}
Let $M$ be a von Neumann algebra and $p, q\in \P(M)$ be two projections with $p\wedge q = 0$. 
Then the following are equivalent. 
\begin{enumerate}
\item The projection $p$ is LS-orthogonal to $q$. 
\item There exists a lattice automorphism $\Phi$ of $\P(M)$ such that $\Phi(p)\perp\Phi(q)$. 
\item If a projection $p_0\in \P(M)$ satisfies $p_0\leq p$ and $p_0\vee q = p\vee q$, then $p_0=p$. 
\item The projection $q$ is LS-orthogonal to $p$. 
\end{enumerate}
\end{lemma}
\begin{proof}
$(1)\Rightarrow (2)$ Take $S\in LS(M)$ as in the preceding lemma and let $\Phi$ be the unique lattice isomorphism such that $\Phi(l(x))= l(SxS^{-1})$, $x\in LS(X)$.
 
$(2)\Rightarrow (3)$ Clear. 

$(3)\Rightarrow (1)$ 
We use the decomposition \eqref{eq:halmos}.
By Lemma \ref{inverse}, if $(1)$ does not hold, then there exists an element $d\in M_{p, q +}\setminus \{0\}$ such that $\{x\in M_{p, q +}\mid x\leq b,\,\, x\leq d\} = \{0\}$. 
Take the nonzero spectral projection $p_1:= \chi_{(\lVert d\rVert/2, \lVert d\rVert]}(d)\in \P(M_{p, q})$. 
It follows that 
\begin{equation}\label{Mpq}
\{x\in M_{p, q +}\mid x\leq b,\,\, x\leq p_1\} = \{0\}.
\end{equation} 
Indeed, if $x\in M_{p, q +}$ satisfies $x\leq b$ and $x\leq p_1$, take a positive real number $c$ with $c\leq 1$ and $c\leq \lVert d\rVert/2$, then $cx\leq cb\leq b$ and $cx\leq cp_1\leq d$, hence $cx=0$ and we obtain $x=0$.
Put $p_0 := 1\oplus 0\oplus 0\oplus \begin{pmatrix}
1-p_1&0 \\
0&0
\end{pmatrix}\in \P(M)$. 
Then $p_0\leq p$ and $p_0\neq p$. 
We prove that $p_0\vee q = p\vee q$, or equivalently, $\begin{pmatrix}
1-p_1&0 \\
0&0
\end{pmatrix} \vee \begin{pmatrix}
a^2&ab \\
ab&b^2
\end{pmatrix} = 1_{\mathbb{M}_2(M_{p, q})}$, which is in turn equivalent to 
\begin{equation}\label{0}
\begin{pmatrix}
p_1&0 \\
0&1
\end{pmatrix} \wedge \begin{pmatrix}
b^2&-ab \\
-ab&a^2
\end{pmatrix} = 0_{\mathbb{M}_2(M_{p, q})}. 
\end{equation}
We have
\[
\begin{split}
&\begin{pmatrix}
1&0 \\
0&0
\end{pmatrix}\left(\begin{pmatrix}
p_1&0 \\
0&1
\end{pmatrix} \wedge \begin{pmatrix}
b^2&-ab \\
-ab&a^2
\end{pmatrix}\right)\begin{pmatrix}
1&0 \\
0&0
\end{pmatrix}\\
&\leq \begin{pmatrix}
1&0 \\
0&0
\end{pmatrix}\begin{pmatrix}
p_1&0 \\
0&1
\end{pmatrix}\begin{pmatrix}
1&0 \\
0&0
\end{pmatrix} = \begin{pmatrix}
p_1&0 \\
0&0
\end{pmatrix}
\end{split}
\]
and 
\[
\begin{split}
&\begin{pmatrix}
1&0 \\
0&0
\end{pmatrix}\left(\begin{pmatrix}
p_1&0 \\
0&1
\end{pmatrix} \wedge \begin{pmatrix}
b^2&-ab \\
-ab&a^2
\end{pmatrix}\right)\begin{pmatrix}
1&0 \\
0&0
\end{pmatrix}\\
&\leq \begin{pmatrix}
1&0 \\
0&0
\end{pmatrix}\begin{pmatrix}
b^2&-ab \\
-ab&a^2
\end{pmatrix}\begin{pmatrix}
1&0 \\
0&0
\end{pmatrix} = \begin{pmatrix}
b^2&0 \\
0&0
\end{pmatrix}.
\end{split}
\]
Since the square root mapping preserves the order of positive operators, \eqref{Mpq} implies that the square root of the operator 
\[
\begin{pmatrix}
1&0 \\
0&0
\end{pmatrix}\left(\begin{pmatrix}
p_1&0 \\
0&1
\end{pmatrix} \wedge \begin{pmatrix}
b^2&-ab \\
-ab&a^2
\end{pmatrix}\right)\begin{pmatrix}
1&0 \\
0&0
\end{pmatrix}
\]
is equal to $0$. Hence
\[
\begin{pmatrix}
1&0 \\
0&0
\end{pmatrix}\left(\begin{pmatrix}
p_1&0 \\
0&1
\end{pmatrix} \wedge \begin{pmatrix}
b^2&-ab \\
-ab&a^2
\end{pmatrix}\right)\begin{pmatrix}
1&0 \\
0&0
\end{pmatrix}=0, 
\]
or equivalently, 
$\begin{pmatrix}
p_1&0 \\
0&1
\end{pmatrix} \wedge \begin{pmatrix}
b^2&-ab \\
-ab&a^2
\end{pmatrix}\leq \begin{pmatrix}
0&0 \\
0&1
\end{pmatrix}$ holds. However, we know $\begin{pmatrix}
0&0 \\
0&1
\end{pmatrix}\wedge\begin{pmatrix}
b^2&-ab \\
-ab&a^2
\end{pmatrix}  = 0$, so we finally obtain \eqref{0}.

Exchanging the roles of $p$ and $q$, we also obtain $(2)\Leftrightarrow(4)$.
\end{proof}

Let us recall the setting of Theorem A: 
Let $M$, $N$ be von Neumann algebras. 
Suppose that $M$ is without type I$_1$ and I$_2$ direct summands and $\Phi\colon \P(M)\to \P(N)$ is a lattice isomorphism. 
By the preceding lemma, we see that $\Phi$ preserves LS-orthogonality in both directions, that is, for any $p, q\in \P(M)$, $p$ and $q$ are LS-orthogonal if and only if $\Phi(p)$ and $\Phi(q)$ are LS-orthogonal. 

In what follows, we show the existence of $\Psi$ as in the statement of Theorem A in the case $M$ has order $3$. 
Thus $M$ can be identified with $\M_3(\hat{M})$ for some von Neumann algebra $\hat{M}$. 
Put 
\[
e_1^M :=\begin{pmatrix}
1&0&0 \\
0&0&0 \\
0&0&0
\end{pmatrix},\,\, 
e_2^M :=\begin{pmatrix}
0&0&0 \\
0&1&0 \\
0&0&0
\end{pmatrix},\,\, 
e_3^M :=\begin{pmatrix}
0&0&0 \\
0&0&0 \\
0&0&1
\end{pmatrix} \in \P(\M_3(\hat{M})). 
\]

Put $e_1:=\Phi(e_1^M)$, $e_2:=\Phi(e_2^M)$, $e_3:=\Phi(e_3^M)$. 
We know that $e_1\vee e_2$ is LS-orthogonal to $e_3$, and $e_1$ is LS-orthogonal and $e_2$. 
In addition, we know $e_1\vee e_2\vee e_3 = 1$. 
Take $S_{e_1\vee e_2, e_3}$ and $S_{e_1, e_2}$ as in the statement of Lemma \ref{orthogonalize}. 
Consider the lattice automorphism $\varphi\colon \P(N)\to \P(N)$ determined by the condition $\varphi(l(x))= l(S_{e_1, e_2}S_{e_1\vee e_2, e_3} x S_{e_1\vee e_2, e_3}^{-1}S_{e_1, e_2}^{-1})\,(=l(S_{e_1, e_2}S_{e_1\vee e_2, e_3} x))$, $x\in LS(N)$. 
Then a moment's calculation shows that $\varphi(e_1), \varphi(e_2), \varphi(e_3)$ are mutually orthogonal and $\varphi(e_1)+ \varphi(e_2)+ \varphi(e_3) = 1_N$. 

\begin{lemma}\label{sim}
We have $\varphi(e_1)\sim \varphi(e_2)\sim \varphi(e_3)$ in $N$. 
\end{lemma}
\begin{proof}
Subsection \ref{halmos} implies: 
For $p, q\in \P(N)$, if $p\vee q = 1$ and $p\wedge q = 0$, then $p^{\perp}\sim q$.
Since $\varphi\circ\Phi$ is a lattice isomorphism, we obtain 
\[
\varphi(e_1) = \varphi\circ\Phi(e_1^M) \sim \left(\varphi\circ\Phi\left(\frac{1}{2}\begin{pmatrix}
1&1&0 \\
1&1&0 \\
0&0&2
\end{pmatrix}\right)\right)^{\perp}\sim \varphi\circ\Phi(e_2^M) = \varphi(e_2). 
\]
Similarly, we obtain $\varphi(e_1)\sim \varphi(e_3)$. 
\end{proof}

It suffices to consider $\varphi\circ \Phi$ instead of $\Phi$. 
Hence we may identify $N$ with $\M_3(\hat{N})$ for some von Neumann algebra $\hat{N}$, and we may assume $\Phi(e_1^M)=e_1^N$, $\Phi(e_2^M)=e_2^N$ and $\Phi(e_3^M)=e_3^N$, where
\[
e_1^N := 
\begin{pmatrix}
1&0&0 \\
0&0&0 \\
0&0&0
\end{pmatrix},\,\, 
e_2^N :=
\begin{pmatrix}
0&0&0 \\
0&1&0 \\
0&0&0
\end{pmatrix},\,\, 
e_3^N :=
\begin{pmatrix}
0&0&0 \\
0&0&0 \\
0&0&1
\end{pmatrix} \in \P(\M_3(\hat{N})). 
\]

Let $x\in LS(\hat{M})$. 
Suppose that $\hat{M}\subset B(K)$. 
Viewing $x$ as a closed operator, we see that the collection 
\[
\left\{\begin{pmatrix}
\xi \\
x\xi \\
0
\end{pmatrix}\in K\oplus K\oplus K \,\middle|\, \xi\in \operatorname{dom}x\right\} 
\]
is a closed subspace in $K\oplus K\oplus K$. 
Take the projection $P_{12}[x]\in \P(B(K\oplus K\oplus K))$ onto this subspace. 
Then we have
\begin{equation}\label{eq:proj}
P_{12}[x] = \begin{pmatrix}
(1+x^*x)^{-1}&(1+x^*x)^{-1}x^*&0 \\
x(1+x^*x)^{-1}&x(1+x^*x)^{-1}x^*&0 \\
0&0&0
\end{pmatrix}
\end{equation}
and hence $P_{12}[x]\in \P(\M_3(\hat{M}))$. 
Similarly, let $P_{13}[x], P_{23}[x]\in \P(\M_3(\hat{M}))$ denote the projections onto 
\[
\left\{\begin{pmatrix}
\xi \\
0 \\
x\xi
\end{pmatrix}\in K\oplus K\oplus K \,\middle|\, \xi\in \operatorname{dom}x\right\},\quad 
\left\{\begin{pmatrix}
0 \\
\xi \\
x\xi
\end{pmatrix}\in K\oplus K\oplus K \,\middle|\, \xi\in \operatorname{dom}x\right\}, 
\]
respectively. 

\begin{lemma}\label{P_x}
Let $Q\in \P(\M_3(\hat{M}))$. 
Then the following conditions are equivalent: 
\begin{enumerate}
\item There exists an $x\in LS(\hat{M})$ such that $Q=P_{12}[x]$. 
\item $Q\vee e_2^M = e_1^M\vee e_2^M$, and $Q$ is LS-orthogonal to $e_2^M$.
\end{enumerate}
\end{lemma}
\begin{proof}
$(1)\Rightarrow (2)$
Let $Q=P_{12}[x]$.
Since $(1+x^*x)^{-1}$ is a positive injective operator, we have $Q\vee e_2^M = e_1^M\vee e_2^M$ by \eqref{eq:proj}. 
Let $x=v\lvert x\rvert$ be the polar decomposition. 
By \eqref{eq:proj}, we have
\[
Q=P_{12}[x] = \begin{pmatrix}
(1+\lvert x\rvert^2)^{-1}&(1+\lvert x\rvert^2)^{-1}\lvert x\rvert v^*&0 \\
v\lvert x\rvert(1+\lvert x\rvert^2)^{-1}&v\lvert x\rvert(1+\lvert x\rvert^2)^{-1}\lvert x\rvert v^*&0 \\
0&0&0
\end{pmatrix}. 
\]
Hence we have 
\[
Q\wedge e_2^M \leq \begin{pmatrix}
0&0&0 \\
0&v\lvert x\rvert(1+\lvert x\rvert^2)^{-1}\lvert x\rvert v^*&0 \\
0&0&0
\end{pmatrix}. 
\]
Since $1-v\lvert x\rvert(1+\lvert x\rvert^2)^{-1}\lvert x\rvert v^*$ is a positive injective operator, we see that $Q\wedge e_2^M=0$. 
As in \eqref{eq:halmos}, we may decompose $e_2^M$ and $Q$ in the following form: 
\[
e_2^M= 1\oplus 0\oplus 0\oplus
\begin{pmatrix}
1&0\\
0&0
\end{pmatrix},\,\, 
Q= 0\oplus 1\oplus 0\oplus
\begin{pmatrix}
a^2&ab \\
ab&b^2
\end{pmatrix}. 
\]
We also have 
\[
e_1^M= 0\oplus 1\oplus 0\oplus
\begin{pmatrix}
0&0\\
0&1
\end{pmatrix} 
\]
with respect to the same decomposition. 
Recall that $(1+x^*x)^{-1}$ is invertible in $LS(\hat{M})$, or equivalently, 
$e_1^MQe_1^M$ is invertible in $LS(e_1^M M e_1^M)$. 
This means that  
\[
0\oplus 1\oplus 0\oplus
\begin{pmatrix}
0&0 \\
0&b^2
\end{pmatrix}
\]
is invertible in $LS(e_1^M M e_1^M)$, which in particular implies the invertibility of $b$ in $LS(M_{e_2^M, Q})$.\\
$(2)\Rightarrow (1)$
As in \eqref{eq:halmos}, we may decompose $e_2^M$ and $Q$ in the following form: 
\[
e_2^M= 1\oplus 0\oplus 0\oplus
\begin{pmatrix}
1&0\\
0&0
\end{pmatrix},\,\, 
Q= 0\oplus 1\oplus 0\oplus
\begin{pmatrix}
a^2&ab \\
ab&b^2
\end{pmatrix}. 
\]
Note that $b$ is invertible as a locally measurable operator. 
Consider the partial isometry
\[
w= 0\oplus 1\oplus 0\oplus
\begin{pmatrix}
0& 0\\
a&b
\end{pmatrix}. 
\]
We have $ww^*=e_1^M$ and $w^*w=Q$. 
A moment's reflection shows that there exist $w_1, w_2\in \hat{M}$ such that $w_1\geq 0$, $w_1$ is invertible in $LS(\hat{M})$ and 
$w=\begin{pmatrix}
w_1&w_2&0 \\
0&0&0 \\
0&0&0
\end{pmatrix}\in \M_3(\hat{M})$.
Put $x= w_2^*w_1^{-1}$. 
Since $ww^*=e_1^M$, we obtain $w_1^2 + w_2w_2^*=1_{\hat{M}}$.
Hence
\[
1+x^*x= 1+ w_1^{-1}w_2w_2^*w_1^{-1} = 1+ w_1^{-1} (1-w_1^2) w_1^{-1} = w_1^{-2}. 
\]
It follows by \eqref{eq:proj} that 
\[
P_{12}[x] = \begin{pmatrix}
(1+x^*x)^{-1}&(1+x^*x)^{-1}x^*&0 \\
x(1+x^*x)^{-1}&x(1+x^*x)^{-1}x^*&0 \\
0&0&0
\end{pmatrix}
= \begin{pmatrix}
w_1^2&w_1w_2&0 \\
w_2^*w_1&w_2^*w_2&0 \\
0&0&0
\end{pmatrix}= w^*w=Q.
\]
\end{proof}

\begin{corollary}
Let $k\in \{12, 13, 23\}$. 
There exists a bijection $\psi_{k}\colon LS(\hat{M})\to LS(\hat{N})$ such that $\Phi(P_k[x]) = P_k[\psi_k(x)]$.
Moreover, $x\in LS(\hat{M})$ is invertible in $LS(\hat{M})$ if and only if $\psi_{k}(x)$ is invertible in $LS(\hat{N})$
\end{corollary}
\begin{proof}
Since $\Phi$ is a lattice isomorphism with $\Phi(e_1^M)= e_1^N$ and  $\Phi(e_2^M)= e_2^N$, the first half of the case $k=12$ follows from the preceding lemma.
For $x\in LS(\hat{M})$, let $P_{21}[x]\in \P(\M_3(\hat{M}))$ denote the projection onto 
\[
\left\{\begin{pmatrix}
x\xi \\
\xi \\
0
\end{pmatrix}\in K\oplus K\oplus K \,\middle|\, \xi\in \operatorname{dom}x\right\}, 
\]
thus
\[
P_{21}[x] = \begin{pmatrix}
x(1+x^*x)^{-1}x^*&x(1+x^*x)^{-1}&0 \\
(1+x^*x)^{-1}x^*&(1+x^*x)^{-1}&0 \\
0&0&0
\end{pmatrix}.
\]
It is easy to  see that, for $x, y\in LS(\hat{M})$, the equation $P_{12}[x]=P_{21}[y]$ holds if and only if $x$ is invertible in $LS(\hat{M})$ and $y=x^{-1}$. 
Therefore, Lemma \ref{P_x} implies that an operator $x\in LS(\hat{M})$ is invertible in $LS(\hat{M})$ if and only if $P_{12}[x]$ is LS-orthogonal to $e_1^M$ and $P_{12}[x]\vee e_1^M = e_1^M\vee e_2^M$. 
Thus $\psi_{12}$ preserves invertibility.
The other cases can be shown similarly. 
\end{proof}

In particular, the operators $\psi_{12}(1), \psi_{13}(1)$ are invertible in $LS(\hat{N})$. 
Consider the lattice automorphism $\phi$ of $\P(\M_3(\hat{N}))$ determined by 
$\phi(l(x)) = l(SxS^{-1})$, where $S= \begin{pmatrix}
1&0&0 \\
0&\psi_{12}(1)^{-1}&0 \\
0&0&\psi_{13}(1)^{-1}
\end{pmatrix}$.
We see that $\phi(e_i^N) = e_i^N$, $i=1, 2, 3$, and 
\[
\phi\circ \Phi(P_{12}[1_{\hat{M}}]) = P_{12}[1_{\hat{N}}],\,\, \phi\circ \Phi(P_{13}[1_{\hat{M}}]) = P_{13}[1_{\hat{N}}].
\]
Considering $\phi\circ \Phi$ instead of $\Phi$, we may assume
$\Phi(P_{12}[1_{\hat{M}}]) = P_{12}[1_{\hat{N}}]$ and $\Phi(P_{13}[1_{\hat{M}}]) = P_{13}[1_{\hat{N}}]$, 
or equivalently, $\psi_{12}(1)=\psi_{13}(1)=1$.

\begin{lemma}\label{multiplicative}
For any $x, y\in LS(\hat{M})$, we have 
\[
P_{13}[xy]=
(P_{23}[-x] \vee P_{12}[y])\wedge (e_1^M\vee e_3^M). 
\]
\end{lemma}
\begin{proof}
Let $\hat{M}\subset B(K)$. 
We know that the range of $P_{23}[-x] \vee P_{12}[y]$ is the closure of 
\[
V:=\left\{\begin{pmatrix}
\eta \\
\xi+ y\eta \\
-x\xi
\end{pmatrix}\in K\oplus K\oplus K \,\middle|\, \xi\in \operatorname{dom}x,\,\, \eta\in \operatorname{dom}y\right\}. 
\]
In particular, we have
$\begin{pmatrix}
\eta \\
0 \\
xy\eta
\end{pmatrix}\in V$ for any $\eta\in \operatorname{dom} y$ with $y\eta\in \operatorname{dom} x$. 
Since the collection $\{\eta\in \operatorname{dom} y\mid y\eta\in \operatorname{dom} x\}$ is a core of the operator $xy\in LS(\hat{M})$, we have $P_{13}[xy]\leq (P_{23}[-x] \vee P_{12}[y])\wedge (e_1^M\vee e_3^M).$  

We claim that the orthogonal complement $V^{\perp}$ of $V$ is 
\[
\left\{\begin{pmatrix}
-y^*x^*\zeta  \\
x^*\zeta \\
\zeta
\end{pmatrix}\in K\oplus K\oplus K \,\middle|\, \zeta\in \operatorname{dom}x^*,\,\, x^*\zeta\in \operatorname{dom}y^*\right\}. 
\]
It is clear that any $\begin{pmatrix}
-y^*x^*\zeta  \\
x^*\zeta \\
\zeta
\end{pmatrix}$ as above is an element in $V^{\perp}$. 
If $\begin{pmatrix}
\zeta_1  \\
\zeta_2 \\
\zeta_3
\end{pmatrix}\in V^{\perp}$, then
\[
0 = \left\langle \begin{pmatrix}
\zeta_1  \\
\zeta_2 \\
\zeta_3
\end{pmatrix}, \begin{pmatrix}
0 \\
\xi \\
-x\xi
\end{pmatrix}\right\rangle = \langle\zeta_2, \xi\rangle - \langle\zeta_3, x\xi\rangle
\]
for any $\xi\in \operatorname{dom}x$, and hence we obtain $\zeta_3\in \operatorname{dom}x^*, \zeta_2= x^*\zeta_3$. 
By the equation
\[
0 = \left\langle \begin{pmatrix}
\zeta_1  \\
\zeta_2 \\
\zeta_3
\end{pmatrix}, \begin{pmatrix}
\eta \\
y\eta \\
0
\end{pmatrix}\right\rangle  
\]
for $\eta\in \operatorname{dom}y$, we obtain the claim. 
Let $\begin{pmatrix}
h_1  \\
h_2 \\
h_3
\end{pmatrix}$ belong to the range of $(P_{23}[-x] \vee P_{12}[y])\wedge (e_1^M\vee e_3^M)$, which is equal to the orthogonal complement of 
$V^{\perp}\cup \left\{\begin{pmatrix}
0\\
k\\
0
\end{pmatrix}\in K\oplus K\oplus K \,\middle|\, k\in K\right\}$. 
Then we have $h_2=0$ and
\[
0 = \left\langle \begin{pmatrix}
h_1  \\
0 \\
h_3
\end{pmatrix}, \begin{pmatrix}
-y^*x^*\zeta  \\
x^*\zeta \\
\zeta
\end{pmatrix}\right\rangle = -\langle h_1, y^*x^*\zeta\rangle + \langle h_3, \zeta\rangle
\]
For any $\zeta\in \operatorname{dom}x^*$ with $x^*\zeta\in \operatorname{dom}y^*$. 
We know that $\{\zeta\in \operatorname{dom}x^*\mid x^*\zeta\in \operatorname{dom}y^*\}$ is a core of the operator $y^*x^*\in LS(\hat{M})$. 
Thus we obtain $h_1\in \operatorname{dom}(y^*x^*)^* = \operatorname{dom}(xy)$  and $h_3 = (xy) h_1$ (here we view $xy$ as a closed operator in $LS(\hat{M})$). 
\end{proof}

\begin{lemma}
We have $\psi_{12}=\psi_{13}=\psi_{23}=: \psi$. 
Moreover, $\psi\colon LS(\hat{M}) \to LS(\hat{N})$ is multiplicative. 
\end{lemma}
\begin{proof}
Let $x, y\in LS(\hat{M})$. 
By the preceding lemma, we have 
\[
P_{13}[xy]=
(P_{23}[-x] \vee P_{12}[y])\wedge (e_1^M\vee e_3^M)
\]
and hence 
\[
\begin{split}
P_{13}[\psi_{13}(xy)]=\Phi(P_{13}[xy]) &=
\Phi\left((P_{23}[-x] \vee P_{12}[y])\wedge (e_1^M\vee e_3^M)\right)\\
&=
(\Phi(P_{23}[-x]) \vee \Phi(P_{12}[y]))\wedge (\Phi(e_1^M)\vee \Phi(e_3^M))\\
&=
(P_{23}[\psi_{23}(-x)] \vee P_{12}[\psi_{12}(y)])\wedge (e_1^N\vee e_3^N).
\end{split}
\]
It follows by the preceding lemma again (applied to $N$ instead of $M$) that 
\[
(P_{23}[\psi_{23}(-x)] \vee P_{12}[\psi_{12}(y)])\wedge (e_1^N\vee e_3^N) = P_{13}[-\psi_{23}(-x)\psi_{12}(y)].
\]
Thus we obtain $P_{13}[-\psi_{23}(-x)\psi_{12}(y)]= P_{13}[\psi_{13}(xy)]$, which implies $-\psi_{23}(-x)\psi_{12}(y) = \psi_{13}(xy)$.

In particular, 
putting $x=y=1$, we obtain $\psi_{23}(-1)= -1$. 
Putting $x=1$, we obtain $-\psi_{23}(-1)\psi_{12}(y) = \psi_{13}(y)$, hence $\psi_{12}(y) = \psi_{13}(y)$. 
Moreover, putting $y=1$, we obtain $-\psi_{23}(-x)\psi_{12}(1) = \psi_{13}(x)$, hence $-\psi_{23}(-x) =\psi_{13}(x)$. 
Thus $\psi_{12}(x) \psi_{12}(y) = -\psi_{23}(-x)\psi_{12}(y) = \psi_{13}(xy) =\psi_{12}(xy)$. 
Therefore, $\psi_{12}$ is multiplicative. 
It follows that $\psi_{12}(-1)$ is central in $LS(\hat{N})$, $\psi_{12}(-1)^2=1$ and $\psi_{12}(-1)y\neq y$ for any $y\neq 0$, and hence we obtain $\psi_{12}(-1)=-1$. 
We reach the equation $\psi_{13}=\psi_{12}=\psi_{23}$. 
\end{proof}

\begin{lemma}
The mapping $\psi$ is additive. 
\end{lemma}
\begin{proof}
Let $x, y\in LS(\hat{M})$. 
Consider the projections 
\[
f= (P_{12}[x] \vee e_3^M) \wedge (P_{13}[1] \vee e_2^M)\quad\text{and}\quad 
g= (P_{12}[y] \vee P_{13}[1]) \wedge (e_2^M\vee e_3^M). 
\]
By an argument similar to that in the proof of Lemma \ref{multiplicative}, we can check the following:  
The range of $f$ is equal to 
\[
\left\{\begin{pmatrix}
\xi \\
x\xi\\
\xi
\end{pmatrix}\in K\oplus K\oplus K \,\middle|\, \xi\in \operatorname{dom}x\right\} 
\]
and the range of $g$ is equal to 
\[
\left\{\begin{pmatrix}
0 \\
-y\eta\\
\eta
\end{pmatrix}\in K\oplus K\oplus K \,\middle|\, \eta\in \operatorname{dom}y\right\}, 
\]
hence $(f\vee g) \wedge (e_1^M\vee e_2^M) = P_{12}[x+y]$. 
Apply $\Phi$ to both sides to obtain the desired conclusion. 
\end{proof}

We define a mapping $\Psi\colon LS(\M_3(\hat{M}))\to LS(\M_3(\hat{N}))$ by $\Psi((x_{ij})_{ij}) := (\psi(x_{ij}))_{ij}$, $x_{ij}\in LS(\hat{M})$, $i, j = 1, 2, 3$. 
The preceding lemmas implies that $\Psi$ is a ring isomorphism from $LS(\M_3(\hat{M}))$ onto $LS(\M_3(\hat{N}))$.

\begin{lemma}
We have $\Phi(l(x)) = l(\Psi(x))$ for any $x\in LS(\M_3(\hat{M}))$. 
\end{lemma}
\begin{proof}
We partly imitate Dye's argument in the proof of \cite[Lemma 7]{Dy}. 
By Lemma \ref{id}, it suffices to show that the lattice isomorphism $\Phi'\colon \P(\M_3(\hat{M})) \to \P(\M_3(\hat{N}))$ determined by $l(\Psi(x)) = \Phi'(l(x))$, $x\in LS(\M_3(\hat{M}))$, satisfies $\Phi=\Phi'$. 
For $x\in LS(\hat{M})$, we have
\[
\begin{split}
\Phi(P_{12}[x]) = P_{12}[\psi(x)] &= l\begin{pmatrix}
1&0&0 \\
\psi(x)&0&0 \\
0&0&0
\end{pmatrix}\\
&= l\left(\Psi \begin{pmatrix}
1&0&0 \\
x&0&0 \\
0&0&0
\end{pmatrix}\right)
= \Phi'\left(l\begin{pmatrix}
1&0&0 \\
x&0&0 \\
0&0&0
\end{pmatrix}\right)= \Phi'(P_{12}[x]). 
\end{split}
\]
Similarly, we see that $\Phi(p)=\Phi'(p)$ for any $p\in \{P_k[x] \mid x\in LS(\hat{M}),\,\, k=12,\,23,\,13\}$. 

Let $x_2, x_3\in LS(\hat{M})$. 
We consider the projection $P_{x_2, x_3}\in \P(\M_3(\hat{M}))$ onto the closed  subspace
\[
\left\{\begin{pmatrix}
\xi \\
x_2\xi \\
x_3\xi
\end{pmatrix}\in K\oplus K\oplus K \,\middle|\, \xi\in \operatorname{dom}x_2\cap \operatorname{dom}x_3\right\}. 
\]
It is not difficult to see that this projection is equal to $(P_{12}[x_2] \vee e_3^M)\wedge (P_{13}[x_3] \vee e_2^M)$.
It follows that $\Phi(P_{x_2, x_3}) = \Phi'(P_{x_2, x_3})$. 

Consider an arbitrary nonzero projection $p = (p_{i, j})_{1\leq i, j\leq 3}\in \P(\M_3(\hat{M}))$. 
By Zorn's lemma, to show that $\Phi(p) = \Phi'(p)$, it suffices to find a nonzero subprojection $(p\geq)\,\, q\in \P(\M_3(\hat{M}))$ such that $\Phi(q) = \Phi'(q)$. 
Note that $p_{ii} = \sum_{1\leq k\leq 3} p_{ik} p_{ik}^*$, hence we see that $p_{ii}\neq 0$ for some $i\in \{1, 2, 3\}$. 

If $p_{11}\neq 0$, put $e:= \chi_{(\lVert p_{11}\rVert/2, \lVert p_{11}\rVert]}(p_{11})\in \P(\hat{M})\setminus \{0\}$ and $x_1:= p_{11}^{-1}e\in \hat{M}$.  
It follows that the projection $q\in \P(\M_3(\hat{M}))$ onto the subspace 
\[
\left\{\begin{pmatrix}
p_{11}\xi \\
p_{21}\xi \\
p_{31}\xi
\end{pmatrix}\in K\oplus K\oplus K \,\middle|\, \xi\in eK\right\} = \left\{\begin{pmatrix}
\eta \\
p_{21}x_1\eta \\
p_{31}x_1\eta
\end{pmatrix}\in K\oplus K\oplus K \,\middle|\, \eta\in eK\right\}
\]
is a nonzero subprojection of $p$. 
Since $q = P_{p_{21}x_1, p_{31}x_1} \wedge ((P_{12}[e^{\perp}]\wedge e_1^M)\vee e_2^M \vee e_3^M)$, we obtain $\Phi(q) = \Phi'(q)$. 

If $p_{11}= 0$ and $p_{22}\neq 0$, we have $p\leq e_2^M\vee e_3^M$. 
Then a similar discussion applies. 
If $p_{11}=p_{22} = 0$, then $p_{33}\in \P(\hat{M})$. 
Use the equation 
$(P_{13}[1]\vee P_{13}[p_{33}^{\perp}])\wedge e_3^M= p$, which can be verified easily, to obtain the desired conclusion. 
\end{proof}

Therefore, the proof of Theorem A is complete in the case $M$ has order $3$.
The same discussion with a slight modification is valid in any case $M$ has order $n$ with $3\leq n<\infty$. 
We know that a projection lattice isomorphism preserves central projections because a projection $p$ in a von Neumann algebra $M$ is central if and only if $\{q\in \P(M)\mid p\vee q=1, p\wedge q=0\} =\{p^{\perp}\}$.
Since every von Neumann algebra without type I$_1$ and I$_2$ direct summands decomposes into the direct sum of algebras of order $3\leq n<\infty$, now it easy to complete the proof of Theorem A in the general case.\medskip

In what follows, let us give a proof of Theorem \ref{dye} by Dye (in the case the von Neumann algebras are without commutative direct summands) as an application of Theorem A. 
The proof below is partly based on Feldman's argument in \cite[Proof of Theorem 3]{Fe}.

Let $M$ and $N$ be von Neumann algebra without type I$_1$ and I$_2$ direct summands and suppose that $\Phi\colon \P(M)\to\P(N)$ is a lattice isomorphism. 
Suppose further that we have $pq = 0$ if and only if $\Phi(p)\Phi(q)=0$ for any pair $p, q\in \P(M)$. 
By Theorem A, there exists a unique ring isomorphism $\Psi\colon LS(M)\to LS(N)$ such that $\Phi(l(x))=l(\Psi(x))$ for any $x\in LS(M)$. 

Then we have $\Psi(p)=\Phi(p)\in \P(N)$ for every $p\in\P(M)$. 
Indeed, since $p^2=p$ and $pp^{\perp}=0$, we have $\Psi(p)^2=\Psi(p)$ and $\Psi(p)\Psi(p^{\perp}) = 0$. 
Thus we have $r(\Psi(p))l(\Psi(p^{\perp})) = 0$.
Our assumption implies $l(\Psi(p^{\perp})) =\Phi(p^{\perp})=\Phi(p)^{\perp}= l(\Psi(p))^{\perp}$, thus we obtain $r(\Psi(p))\leq l(\Psi(p))$. 
By the equation $(l(\Psi(p))-\Psi(p))\Psi(p)=0$, we obtain $0=(l(\Psi(p))-\Psi(p))l(\Psi(p)) = l(\Psi(p))-\Psi(p)$.
Hence $\Psi(p)=l(\Psi(p)) =\Phi(p)\in \P(N)$.

Consider the ring automorphism $x\mapsto \Psi^{-1}(\Psi(x^*)^*)$ of $LS(M)$. 
This fixes every projection, hence Lemma \ref{id} implies that $x= \Psi^{-1}(\Psi(x^*)^*)$, or equivalently, $\Psi(x)^*=\Psi(x^*)$ for each $x\in LS(M)$.
It follows that $\Psi$ maps the self-adjoint part of $LS(M)$ onto that of $LS(N)$. Since $\Psi$ preserves squares, $\Psi$ restricted to self-adjoint parts preserves order in both directions. 
Since $\Psi(1)=1$, $\Psi$ restricts to a real $^*$-isomorphism from $M$ onto $N$ and extends $\Phi$, which is the desired conclusion.

\section{Ring isomorphisms of locally measurable operator algebras}\label{extra}
By the preceding section, lattice isomorphisms between projection lattices are in one-to-one correspondence with ring isomorphisms between the algebras of locally measurable operators. 
Hence the following question is well motivated.
\begin{q}
Let $M$, $N$ be von Neumann algebras. 
What is the general form of ring isomorphisms from $LS(M)$ onto $LS(N)$?
\end{q}

\begin{lemma}
Let $M, N$ be general von Neumann algebras. 
Let 
\[
\begin{split}
M&=\left(\bigoplus_{n\geq 1}M_{\mathrm{I}_n}\right) \oplus M_{\mathrm{I}_{\infty}} \oplus M_{\mathrm{II}_1} \oplus M_{\mathrm{II}_{\infty}} \oplus M_{\mathrm{III}},\\
N&=\left(\bigoplus_{n\geq 1}N_{\mathrm{I}_n}\right) \oplus N_{\mathrm{I}_{\infty}} \oplus N_{\mathrm{II}_1} \oplus N_{\mathrm{II}_{\infty}} \oplus N_{\mathrm{III}}
\end{split}
\]
be the type decompositions, where $M_j$, $N_j$ are von Neumann algebras of type $j$. 
Suppose that $\Psi\colon LS(M)\to LS(N)$ is a ring isomorphism.
Then there exist ring isomorphisms $\psi_j\colon LS(M_j)\to LS(N_j)$ such that $\Psi(x) =\psi_j(x)$ for any $x\in LS(M_j)\,\,(\subset LS(M))$.
\end{lemma}
\begin{proof}
It is easy to see that $\Psi$ maps the collection of central projections in $M$ onto that in $N$. 
Hence it suffices to show: If $M$, $N$ are of type $j, k\in \{\mathrm{I}_n\mid n\geq 1\}\cup \{\mathrm{I}_{\infty}, \mathrm{II}_1, \mathrm{II}_{\infty}, \mathrm{III}\}$, respectively, then $j=k$. 
We consider the lattice isomorphism $\Phi\colon \P(M)\to\P(N)$ as in Proposition \ref{converse}. 
It is easy to see that a projection $p\in \P(M)$ is abelian (namely, $pMp$ is an abelian von Neumann algebra) if and only if $\Phi(p)$ is abelian. 
Moreover, a projection $p\in \P(M)$ is finite if and only if $\Phi(p)\in \P(N)$ is finite. 
Indeed, if $p$ is not finite, then there exist mutually orthogonal nonzero subprojections $p_1, p_2, p_3$ of $p$ such that $p_1\sim p_2\sim p_3\sim p_1+p_2$. 
The same argument as in the proof of Lemma \ref{sim} implies $\Phi(p_1)\sim\Phi(p_3)\sim \Phi(p_1+p_2)$, which shows that $\Phi(p)$ is not finite. 
Similarly, if $\Phi(p)$ is not finite, then $p$ is not finite. 
The rest of the proof is a standard argument of von Neumann algebra theory, and we omit the details.
See e.g.\ \cite[Chapter 6]{KR}.
\end{proof}

Therefore, Question reduces to the case both $M$ and $N$ are of type $j$, $j\in \{\mathrm{I}_n\mid n\geq 1\}\cup \{\mathrm{I}_{\infty}, \mathrm{II}_1, \mathrm{II}_{\infty}, \mathrm{III}\}$. 

We first consider Question in the case $M, N$ are von Neumann algebras of type I$_n$. 
Suppose that $LS(M)$ is ring isomorphic to $LS(N)$. 
Since the central projection lattices of $M$ and $N$ are lattice isomorphic, we see that the center of $M$ is $^*$-isomorphic to that of $N$. 
Hence there exists a commutative von Neumann algebra $A$ such that $M\cong N\cong \M_n (A)$. 
Therefore, it suffices to think about ring automorphisms of $LS(\M_n(A))$, which can be identified with the collection of all $n \times n$ matrices with entries in $LS(A)$. 
Note that $A$ can be identified with the algebra $L^{\infty}(\mu)$ of all complex-valued essentially bounded measurable functions (modulo almost-everywhere equivalence) for some measure $\mu$. 
Then $LS(A)$ corresponds to $L^0(\mu)$, which denotes the collection of all complex-valued measurable functions.
Remark that any ring automorphism $\psi$ of $LS(A)$ determines a ring automorphism $\psi'$ of $LS(\M_n(A))$ by the formula $\psi'((x_{ij})) = (\psi(x_{ij}))_{ij}$. 
The following proposition slightly generalizes (but can be shown by exactly the same argument as in) \cite[Theorem 3.3]{AAKD} by Albeverio, Ayupov, Kudaybergenov and  Djumamuratov. 
\begin{proposition}\label{I}
Let $n\geq 1$ be an integer and $A$ be a commutative von Neumann algebra. 
Suppose that $\Psi$ is a ring automorphism of $LS(\M_n(A))$.
Then there exist a ring automorphism $\psi\colon LS(A)\to LS(A)$ and an invertible element $y$ in $LS(\M_n(A))$ such that $\Psi(x)=y\psi'(x)y^{-1}$, $x\in LS(\M_n(A))$. 
\end{proposition}
\begin{proof}
Note that $\Psi$ restricts to a ring automorphism $\psi$ of the center of $LS(\M_n(A))$, which is canonically isomorphic to $LS(A)$. 
Then $\Psi\circ {\psi'}^{-1}$ fixes every element in the center of $LS(\M_n(A))$.
We may apply \cite[Theorem 3.1]{AAKD} to obtain the desired conclusion. 
\end{proof}

There exist highly nontrivial examples of ring automorphisms of $LS(A) = L^0(\mu)$ for a commutative von Neumann algebra $A$.
For example, consider the case $A=\C=LS(A)$. 
There are many ring automorphisms of $\C$ that are far from real-linear. 
Consider the case $\mu$ is an atomless measure. 
It is known \cite[{(1)$\Leftrightarrow$(6) of Theorem 3.4}]{Ku} (see also \cite[Remark 6.3]{Ku}) that there exists a (complex-linear) algebra automorphism $\psi$ of $L^0(\mu)$ such that $\psi(p)=p$ for any $p\in \P(A)$ but $\psi\neq\operatorname{id}_{L^0(\mu)}$. 
It seems that these examples are beyond the scope of the theory of operator algebras. 

In contrast, we may give a purely operator algebraic solution to Question for type I$_{\infty}$ or III in the following manner.
This improves \cite[Theorem 3.8]{AAKD}, in which algebra isomorphisms of the case of type I$_\infty$ were considered.
\begin{thmb}
Let $M, N$ be von Neumann algebras of type I$_{\infty}$ or III. 
If $\Psi\colon LS(M)\to LS(N)$ is a ring isomorphism, then there exist a real $^*$-isomorphism $\psi\colon M\to N$ (which extends to a real $^*$-isomorphism from $LS(M)$ onto $LS(N)$) and an invertible element $y\in LS(N)$ such that  $\Psi(x)=y\psi(x)y^{-1}$, $x\in LS(M)$.
\end{thmb}
\begin{proof}
Beware of the fact that $\Psi$ restricts to a lattice isomorphism between the central projection lattices of $M$ and $N$. 
We first prove:
\begin{itemize}
\item[\underline{Claim}] 
There exists an operator $a\in LS(\Z(N))_+$ such that $\vertiii{\Psi(x)}\leq a$ for any $x\in M\,\,(\subset LS(M))$ with $\lVert x\rVert\leq 1$. 
\end{itemize}
Assume that this claim does not hold. 
We will obtain a contradiction in Step 4.\smallskip

\noindent
\underline{Step 1} 
We prove that there exists a central projection $e$ in $M$ such that for any $n\geq 1$ there exists some $x\in M$ with $\lVert x\rVert \leq 1$ and $\vertiii{\Psi(x)}\geq n\Psi(e)$. 

Assume for a while that the center $\Z(M)$ of $M$ admits a faithful normal state $\tau\colon \Z(M)\to \C$. 
For each positive integer $n$, consider the collection 
\[
{E}_n := \{ e\in \P(\Z(M))\mid \text{ there exists }x\in M \text{ with } \lVert x\rVert \leq 1 \text{ and } \vertiii{\Psi(x)}\geq n\Psi(e)\}. 
\]
Suppose that $e, f$ belong to this collection. 
Take $x, y \in M$ such that $\lVert x\rVert,  \lVert y\rVert\leq 1$ and $\vertiii{\Psi(x)}\geq n\Psi(e)$, $\vertiii{\Psi(y)}\geq n\Psi(f)$. 
Then the element $x':=xe+ye^{\perp}$ satisfies $\lVert x'\rVert\leq 1$ and 
\[
\begin{split}
\vertiii{\Psi(x')} &= \vertiii{\Psi(xe+ye^{\perp})} \\
&= \vertiii{\Psi(x)\Psi(e)+\Psi(y)\Psi(e)^{\perp}}\\
&=\vertiii{\Psi(x)}\Psi(e) + \vertiii{\Psi(y)}\Psi(e)^{\perp}\\
&\geq n\Psi(e) + n\Psi(f)\Psi(e)^{\perp}\\
&= n\Psi(e)\vee\Psi(f) = n\Psi(e\vee f). 
\end{split}
\]
Hence we have $e\vee f\in {E}_n$, which implies that ${E}_n$ is upward directed. 
Put $c_n:= \sup\{\tau(e)\mid e\in {E}_n\}$. 
We may take an increasing sequence $\{e^{(k)}\}\subset {E}_n$ such that $\tau(e^{(k)})\to c_n$ as $k\to\infty$. 
For each $k$ take $x^{(k)} \in M$ such that $\lVert x^{(k)}\rVert\leq 1$ and $\vertiii{\Psi(x^{(k)})}\geq n\Psi(e^{(k)})$. 
Some calculations show that the element 
\[
x'':= x^{(1)}e^{(1)} + \sum_{k\geq 2} x^{(k)}(e^{(k)}-e^{(k-1)})\in M
\]
satisfies $\lVert x''\rVert\leq 1$ and $\vertiii{\Psi(x'')}\geq n\Psi(e^{(k)})$ for every $k$.
This implies that for the projection $e_n := \bigvee{E}_n\in \P(\Z(M))$ there exists  $x_n\in \P(\Z(M))$ such that $\lVert x_n\rVert\leq 1$ and $\vertiii{\Psi(x_n)}\geq n\Psi(e_n)$. 

Clearly, $\{e_n\}$ is a decreasing sequence. 
Assume that $e_n\to 0$ as $n\to \infty$, then the element $a=\Psi(1+\sum_{n\geq 1}e_n)\in LS(\Z(N))_+$ satisfies the property of Claim, which contradicts our assumption.
Hence we have $e_n\to e\in \P(\Z(M))\setminus\{0\}$ as $n\to \infty$, and $e$ satisfies the desired property. 
Since every von Neumann algebra can be decomposed into the direct sum of von Neumann algebras whose centers admit faithful normal states, the same holds for arbitrary $M$ and $N$.\smallskip

Considering the restriction of $\Psi$ to a ring isomorphism from $LS(Me)$ onto $LS(N\Psi(e))$, we may assume that for any $n\geq 1$ there exists some $x\in M$ with $\lVert x\rVert \leq 1$ and $\vertiii{\Psi(x)}\geq n$.\smallskip

\noindent\underline{Step 2}
We prove that for any $a\in LS(\Z(N))_+$ there exists some $x\in M$ with $\lVert x\rVert \leq 1$ and $\vertiii{\Psi(x)}\geq a$. 

Let $a\in LS(\Z(N))_+$.
We may take a sequence of mutually orthogonal central projections $\{ f_n \}$ such that  
$a\leq \sum_{n\geq 1}nf_n$. 
For each $n$, take $x_n \in M$ such that $\lVert x_n\rVert\leq 1$ and $\vertiii{\Psi(x_n)}\geq nf_n$. 
Some calculations show that the element $x:=  \sum_{n\geq 1} x_n\Psi^{-1}(f_n)$
satisfies $\lVert x\rVert\leq 1$ and $\vertiii{\Psi(x)}\geq \sum_{n\geq 1} nf_n \geq a$.\smallskip

\noindent\underline{Step 3} 
We prove: For any $p\in \P(M)$ with $p\sim p^{\perp}$ and any $a\in LS(\Z(N))_+$, there exists an element $x\in M$ with $pxp=x$, $\lVert x\rVert\leq 1$ and $\vertiii{\Psi(x)}\geq a$. 

Take a partial isometry $v\in M$ such that $vv^*=p$ and $v^*v=p^{\perp}$. 
Since $\Psi$ is a ring isomorphism, for any $x\in M$, we have 
\[
\begin{split}
\Psi(x) &= \Psi(pxp+pxp^{\perp}+p^{\perp}xp+p^{\perp}xp^{\perp})\\
&= \Psi(pxp)+\Psi(pxv^*)\Psi(v)+\Psi(v^*)\Psi(vxp)+\Psi(v^*)\Psi(vxv^*)\Psi(v).
\end{split}
\]
For a given $a\in LS(\Z(N))_+$, put 
\[
b:= 4a+4a \vertiii{\Psi(v)}+ 4a \vertiii{\Psi(v^*)}+ 4a\vertiii{\Psi(v)}
\vertiii{\Psi(v^*)}\in LS(\Z(N))_+.
\]
The preceding step implies there exists $x\in M$ with $\lVert x\rVert\leq 1$ and 
\[
\begin{split}
b&\leq \vertiii{\Psi(x)}\\
&\leq \vertiii{\Psi(pxp)}+\vertiii{\Psi(pxv^*)}\vertiii{\Psi(v)}+\vertiii{\Psi(v^*)}\vertiii{\Psi(vxp)}+\vertiii{\Psi(v^*)}\vertiii{\Psi(vxv^*)}\vertiii{\Psi(v)}.
\end{split}
\]
Hence there exists a quadruple $f_1, f_2, f_3, f_4$ of central projections in $N$ such that $f_1+f_2+f_3+f_4=1$, $\vertiii{\Psi(pxp)}f_1\geq bf_1/4$, $\vertiii{\Psi(pxv^*)}\vertiii{\Psi(v)}f_2\geq bf_2/4$, $\vertiii{\Psi(v^*)}\vertiii{\Psi(vxp)}f_3\geq bf_3/4$ and $\vertiii{\Psi(v^*)}\vertiii{\Psi(vxv^*)}\vertiii{\Psi(v)}f_4\geq bf_4/4$. 
Put 
\[
x' := pxp\Psi^{-1}(f_1) + pxv^*\Psi^{-1}(f_2) + vxp\Psi^{-1}(f_3) + vxv^*\Psi^{-1}(f_4). 
\]
Then we have $px'p=x'$, $\lVert x'\rVert \leq 1$ and 
\[
\begin{split}
\vertiii{\Psi(x')}&= \vertiii{\Psi(pxp)f_1 + \Psi(pxv^*)f_2 + \Psi(vxp)f_3 + \Psi(vxv^*)f_4}\\
&= \vertiii{\Psi(pxp)}f_1 + \vertiii{\Psi(pxv^*)}f_2 + \vertiii{\Psi(vxp)}f_3 +\vertiii{\Psi(vxv^*)}f_4\\
&\geq \frac{1}{4}b (f_1 + \vertiii{\Psi(v)}^{-1}f_2 + \vertiii{\Psi(v^*)}^{-1}f_3 + \vertiii{\Psi(v)}^{-1}\vertiii{\Psi(v^*)}^{-1}f_4)\geq a.  
\end{split}
\]
(Note that $\vertiii{\Psi(v)}$, $\vertiii{\Psi(v^*)}$ are invertible in $LS(\Z(N))$.) 
\smallskip

\noindent\underline{Step 4}
Since $M$ is properly infinite, we may take a sequence $\{p_n\}_{n\geq 1}$ of mutually orthogonal projections in $M$ such that $p_n\sim p_n^{\perp}$, $n\geq 1$. 
By Step 3, for each $n\geq 1$, we may take an element $x_n\in M$ with $p_nx_np_n=x_n$, $\lVert x_n\rVert\leq1$ and $\vertiii{\Psi(x_n)}\geq n\vertiii{\Psi(p_n)}$. 
Put $x:=\sum_{n\geq 1} x_n\in M$ (which is well-defined since $p_n$, $n\geq 1$, are mutually orthogonal). 
For every $n\geq 1$, we have 
\[
\vertiii{\Psi(x)}\vertiii{\Psi(p_n)}\geq\vertiii{\Psi(x)\Psi(p_n)}=\vertiii{\Psi(xp_n)} = \vertiii{\Psi(x_n)}\geq n\vertiii{\Psi(p_n)}. 
\]
Since $\vertiii{\Psi(p_n)}$ is invertible in $LS(\Z(N))$, we obtain $\vertiii{\Psi(x)}\geq n$ for all $n\in \N$, a contradiction.
This completes the proof of Claim.\smallskip

\noindent\underline{Step 5}
It follows that there exists an element $a\in LS(\Z(N))_+$ such that $\vertiii{\Psi(x)}\leq a$ for any $x\in M$ with $\lVert x\rVert\leq 1$. 
By the same discussion applied to $\Psi^{-1}$, we also obtain an element $a'\in LS(\Z(M))_+$ such that $\vertiii{\Psi^{-1}(y)}\leq a'$  for any $y\in N$ with $\lVert y \rVert \leq 1$.  
We may take a sequence $e_n$ of central projections in $M$ such that $e_n\nearrow 1$ and $\Psi$ restricts to a norm-bicontinuous ring isomorphism $\Psi_n$ from $M e_n$ onto $N\Psi(e_n)$, $n\geq 1$.
By Lemma \ref{isoms} we may verify the statement for each $\Psi_n$, which suffices to complete the proof. 
\end{proof}

\begin{corollary}
Let $M, N$ be von Neumann algebras of type I$_{\infty}$ or III. 
Suppose that $\Phi\colon \P(M)\to \P(N)$ is a lattice isomorphism. 
Then there exist a real $^*$-isomorphism $\psi\colon M\to N$ and an invertible element $y\in LS(N)$ such that $\Phi(p)=l(y\psi(p))$, $p\in \P(M)$.
\end{corollary}

\section{Questions}\label{question}
The author skeptically conjectures that the same as Theorem B holds for type II von Neumann algebras:
\begin{conjecture}
Let $M$ and $N$ be von Neumann algebras of type II. 
Suppose that $\Psi\colon LS(M)\to LS(N)$ is a ring isomorphism. 
Then there exist an invertible operator $y\in LS(N)$ and a real $^*$-isomorphism $\psi\colon M\to N$ such that $\Psi(x) = y\psi(x)y^{-1}$ for any $x\in LS(M)$. 
\end{conjecture}

Not much is known about the structure of the algebra $LS(M)$ for a type II (in particular, II$_1$) von Neumann algebra $M$. 
The author does not know whether or not such a $\Psi$ is automatically real-linear even in the case $M$ and $N$ are (say, approximately finite dimensional) II$_1$ factors. 
Note that $LS(M)$ cannot have a Banach algebra structure because of the fact that an element of $LS(M)$ can have an empty or dense spectral set. 
Hence it seems to be difficult to make use of \emph{automatic continuity} results on algebra isomorphisms as in \cite{Da}. 
However, the author suspects that at least the following weaker statement holds: 
\begin{conjecture}
Let $M$ and $N$ be von Neumann algebras of type II. 
If $\P(M)$ and $\P(N)$ are lattice isomorphic, or equivalently, if $LS(M)$ and $LS(N)$ are ring isomorphic, then $M$ and $N$ are real $^*$-isomorphic (or equivalently, $M$ and $N$ are Jordan $^*$-isomorphic).  
\end{conjecture}

In another direction, we compare Theorem A with von Neumann's theory of complemented modular lattices and regular rings. 
Von Neumann axiomatized projection lattices of type II$_1$ von Neumann algebras, and completed the amazing theory on the correspondence between the vast classes of complemented modular lattices and regular rings. 
Let us briefly recall this theory in \cite[Part II]{N}. 
\begin{definition}
A lattice $L$ with greatest element $1$ and least element $0$ is \emph{complemented} if for each $a\in L$ there exists $b\in L$ such that $a\vee b=1$, $a\wedge b=0$. 
A lattice $L$ is \emph{modular} if the equation $(a\vee b)\wedge c=a\vee (b\wedge c)$ holds for any $a, b, c\in L$ with $a\leq c$. 

Let $L$ be a complemented modular lattice.
Two elements $a, b\in L$ are said to be \emph{perspective} if there exists $c\in L$ such that $a\vee c=1=b\vee c$ and $a\wedge c=0=b\wedge c$. 
Let $n$ be a positive integer. 
We say $L$ has \emph{order $n$} if there exist pairwise perspective elements $a_1, a_2, \ldots, a_n\in L$ with $a_1\vee a_2\vee\cdots\vee a_n=1$ and $\bigl(\bigvee_{i\in I_1}a_i\bigr) \wedge \bigl(\bigvee_{j\in I_2}a_j\bigr)=0$ for any disjoint $I_1, I_2\subset \{1, 2, \ldots, n\}$.
\end{definition}

\begin{definition} 
A (\emph{von Neumann}) \emph{regular ring} is a ring $R$ with unit such that for each $x\in R$ there exists $y\in R$ such that $xyx=x$. 

Let $R$ be a regular ring. 
A right ideal $\mathfrak{a}$ of $R$ is \emph{principal} if it is generated by one element of $R$. 
\end{definition}
Let $M$ be a von Neumann algebra. 
Then $\P(M)$ is a complemented lattice. 
It is not difficult to show that the following three conditions are equivalent. 
\begin{itemize}
\item The von Neumann algebra $M$ is finite.
\item The lattice $\P(M)$ is modular. 
\item The ring $LS(M)$ is regular. 
\end{itemize}

\begin{theorem}[von Neumann]
If $R$ is a regular ring, then the collection $L$ of all principal right ideals of $R$, ordered by inclusion, forms a complemented modular lattice. 
\end{theorem}
We call $L$ in the statement of the preceding theorem the \emph{right ideal lattice of $R$}. 
\begin{theorem}[von Neumann]
Let $R_1, R_2$ be regular rings with right ideal lattices $L_1, L_2$, respectively. Suppose that $L_1$ has order $n\geq 3$. If $\Phi\colon L_1\to L_2$ is a lattice isomorphism, then there exists a unique ring isomorphism $\Psi\colon R_1\to R_2$ such that $\Phi(\mathfrak{a})=\Psi(\mathfrak{a})$, $\mathfrak{a}\in L_1$. 
\end{theorem}
\begin{theorem}[von Neumann]
Let $L$ be a complemented modular lattice with order $n\geq 4$. 
Then there exists a regular ring $R$ such that the right ideal lattice of $R$ is lattice isomorphic to $L$. 
\end{theorem}
Let $M$ be a finite von Neumann algebra. 
Let $\mathfrak{a}\subset LS(M)$ be a principal right ideal generated by $a\in LS(M)$. 
It is an easy exercise to show that $\mathfrak{a}=\{x\in LS(M)\mid l(x)\leq l(a)\}$. 
Hence we obtain an identification of the right ideal lattice of $LS(M)$ with the projection lattice $\P(M)$. 
In particular, Theorem \ref{vN} is a corollary of von Neumann's results above. 
See also the article \cite{Go} by Goodearl, which deals with the history of the study of regular rings in connection with functional analysis.

Von Neumann's theory, applied to the setting of von Neumann algebras, is valid only for finite von Neumann algebras. 
In this paper, we proved that there exists a complete correspondence between lattice isomorphisms and ring isomorphisms in the general setting of von Neumann algebras. 
Hence the author believes that one might be able to generalize von Neumann's theory to a broader class of lattices that covers projection lattices of any von Neumann algebras (of fixed order $n\geq 3$ or $4$). 
This is left as a research program in the future.\medskip

\textbf{Acknowledgements} \quad 
The author appreciates Yasuyuki Kawahigashi, who is the advisor of the author, for invaluable support. 
The author is also indebted to Peter \v{S}emrl and Yuhei Suzuki, who encouraged the author to begin working on lattice isomorphisms. 
This work was supported by Leading Graduate Course for Frontiers of Mathematical Sciences and Physics (FMSP) and JSPS Research Fellowship for Young Scientists (KAKENHI Grant Number 19J14689), MEXT, Japan.  
Part of this work was completed during the author's visit to Gilles Pisier in IMJ-PRG, Paris. 
This visit was supported by JSPS Overseas Challenge Program for Young Researchers.

\end{document}